\documentclass[10pt]{amsart}
\usepackage{amssymb}
\usepackage{bm}
\usepackage{graphicx}
\usepackage[centertags]{amsmath}
\usepackage{amsfonts}
\usepackage{amsthm}
\usepackage{comment}

\usepackage{pgf}
\usepackage{pgfkeys}
\usepackage{ifpdf}
\usepackage{color}

\newtheorem{thm}{Theorem}

\newtheorem{cor}[thm]{Corollary}
\newtheorem{lem}[thm]{Lemma}

\newtheorem{defn}[thm]{Definition}
\theoremstyle{definition}
\newtheorem{rem}{Remark}

\newtheorem{fact}[thm]{Fact}

\newcommand{\nn}{\mathbb{N}}

\newcommand{\vs}{\vec{s}}
\newcommand{\vt}{\vec{t}}
\newcommand{\vw}{\vec{w}}
\newcommand{\V}{V^\infty (A)}

\newtheorem*{thm*}{Theorem}

\begin{document}
\title[A Comb. proof of an Infinite version of the  Hales--Jewett theorem]
{A Combinatorial proof of an Infinite version of the  Hales--Jewett
theorem}
\author{Nikolaos Karagiannis}
\address{National Technical University of Athens, Faculty of Applied Sciences,
Department of Mathematics, Zografou Campus, 157 80, Athens, Greece}
\email{nkaragiannis@math.ntua.gr}


\begin{abstract} We  provide a combinatorial proof of an
infinite  extension of the
Hales--Jewett theorem due to T. Carlson and
independently due to H. Furstenberg and Y. Katznelson.
\end{abstract}

\maketitle


\section{Introduction}\label{intro}

The aim of this note is to provide a new combinatorial  proof
of a well-known infinitary
extension of the Hales--Jewett  theorem. To state it, we need first
to recall the relevant terminology.  Let $A$ be an
\textit{alphabet} (i.e. any non-empty  set). The elements of $A$
will be called \textit{letters}. By $W(A)$ we denote the set of
\textit{constant words} \textit{over} $A$, that is the set of all
finite sequences with elements in $A$ including the empty
sequence. For $N \in \nn$, by $A^N$, we denote all finite
sequences from $A$, consisting of $N$ letters. We also fix an
element $x \notin A$ which will be regarded  as a
\textit{variable}. A \textit{variable word} \textit{over} $A$ is
an element in $W(A \cup \{x\}) \setminus W(A)$. The variable words
will be denoted by $s(x), t(x), w(x)$, etc. Given a variable word
$s(x)$ and $a\in A$, by $s(a)$ we denote the constant word in
$W(A)$ resulting from the substitution of  the variable $x$ with
the letter $a$. Let $q\in\nn$ with $q\geq 1$, then a
$q$-\textit{coloring} of a set $X$ is any map $c:X \rightarrow
\{1,...,q\}$. A subset $Y$ of $X$ will be called
\textit{monochromatic}, if there exists $1\leq i \leq q $ such
that $c(y)=i$, for all $y \in Y$. Finally, for every finite set
$X$, by $|X|$ we denote its cardinality.

We recall the following fundamental result in Ramsey Theory, due
to A. Hales and R. Jewett \cite{H-J}.

\begin{thm}\label{thm1} For every positive integers $p,q$
there exists a positive integer
  $HJ(p,q)$ with the following property.
For every finite alphabet $A$
  with $|A|=p$,
every $N \geq HJ(p,q)$ and every $q$-coloring of
  $A ^N$ there exists a variable word $w(x)$ of length $N$ such that the set
  $\{w(a):a \in A\}$ is monochromatic.
\end{thm}

The Hales--Jewett theorem gave birth to a whole new branch of
research concerning extensions of it in the context
of both  finite and infinite alphabets (see \cite{C}, \cite{C-S},
\cite{F-K}, \cite{GR}, \cite{HMC}, \cite{MC}, \cite{PV}). For an exposition
of these results the reader can also refer to \cite{GRS}, \cite{Hb},
 \cite{MCb}, \cite{Tod}.

The  first  theorem that we will prove is  due to T. Carlson
\cite{C} and independently due to H. Furstenberg and Y. Katznelson
\cite{F-K} and is the following.

\begin{thm}\label{thm3f}
  Let  $A$ be a finite alphabet.
 Then for every finite coloring of $W(A)$ there exists a sequence
  $(t_n(x))_{n=0}^ \infty$ of variable words over $A$
  such that  for every $n\in\nn$ and every  $m_0<m_1<...<m_n$,
  the  words of the form $t_{m_0}(a_0) t_{m_1}(a_1) ... t_{m_n}(a_n)$ with
  $a_{i} \in A$ for all $0 \leq i \leq n$ are of the same color.
\end{thm}

Our approach for Theorem \ref{thm3f} can be extended in order to
provide a proof for a
 stronger version of it which concerns
infinite increasing
sequences of finite
alphabets and is the following.

\begin{thm}\label{thm3}
  Let $(A_n)_{n=0}^ \infty$  be an increasing  sequence of finite
  alphabets and let $A=\cup_{n \in \nn} A_n$.
 Then for every finite coloring of $W(A)$ there exists a sequence
  $(t_n(x))_{n=0}^ \infty$ of variable words over $A$
  such that for every $n\in\nn$ and every  $m_0<m_1<...<m_n$,
  the  words of the form $t_{m_0}(a_0) t_{m_1}(a_1) ... t_{m_n}(a_n)$ with
 $a_0 \in A_{m_0}, a_1\in A_{m_1}, ... , a_n \in A_{m_n}$ are of the same color.
\end{thm}

Let us point out that there exist
easy counterexamples that show that a direct extension of Theorem \ref{thm3f}
for an infinite alphabet $A$ is false.
 Theorem \ref{thm3} is a consequence of a more general
result of T. Carlson (see \cite[Theorem 15]{C}). The original proofs
of Theorems \ref{thm3f} and \ref{thm3} are based on topological as
well as  algebraic notions of the Stone--{\v C}ech compactification
of the related structures. Our approach is strictly combinatorial
and relies on the classical Hales--Jewett theorem. It has its
origins in the proof of Hindman's theorem \cite{H} due to J. E.
Baumgartner \cite{B} and  is close in  spirit with the proof of
Carlson--Simpson's theorem \cite{C-S}. In particular, a proof of a
weaker version of Theorem \ref{thm3f}, given by R. McCutcheon in
\cite[\S 2.3]{MCb}, was the motivation for this note.

Clearly, Theorem \ref{thm3f} is a consequence of Theorem
\ref{thm3}. Although  the proofs of both theorems follow similar
arguments, Theorem \ref{thm3} is  more demanding and quite more
technical. For this reason and  in order to make the presentation
more clear, we have decided to start with  a detailed exposition
of Theorem \ref{thm3f} and then proceed to Theorem \ref{thm3}. The
present note is an updated and extended version of part of
\cite{Kar}.

\section{Proof of Theorem \ref{thm3f}}\label{main proof}

\subsection{Preliminaries} \label{Initializing the proof}
In this subsection we introduce some notation and
terminology that we will use for the proof of Theorem \ref{thm3f}.
 For the following, we fix  a finite alphabet $A$. Let
 $V(A)$ be the set of all \textit{variable words} (\textit{over}
 $A$). By $V^ {<\infty}(A) $ (resp. $V^ \infty(A)
$) we denote the set of all finite (resp. infinite) sequences of
variable words. Also let $V^{\leq\infty}(A)=V^ {<\infty}(A) \cup
V^ \infty(A)$. Generally, the elements of $V^{\leq\infty}(A)$ will
be denoted by $\vs,\vt, \vw$, etc. Also by $\nn= \{0,1,...\}$ we
denote the set of all non negative integers.

\subsubsection{Constant and variable  span of a sequence of variable words over  $A$}
 Let  $m\in \nn$ and
$(s_n(x))_{n=0}^m\in V^{<\infty}(A)$.

(a) The \textit{constant span of} $(s_n(x))_{n=0}^m$, denoted by
$<(s_n(x))_{n=0}^m >_c$, is defined to be the set
$$\bigcup_{n=0}^m\big\{s_{l_0}(a_0)...
s_{l_n}(a_n): 0 \leq l_0<...<l_n\leq m,\;
   a_0,  ... , a_n\in A \big\}.$$

(b) The \textit{variable span of}
$(s_n(x))_{n=0}^m$,  denoted by $<(s_n(x))_{n=0}^m
>_v$, is defined to be the set
\[V(A) \cap  \bigcup_{n=0}^m\big\{s_{l_0}(a_0)... s_{l_n}(a_n):
0 \leq l_0<...<l_n\leq m, \;
  a_0, ... , a_n\in A\cup \{x\}
\big\}.\] The above notation is naturally extended to infinite
sequences of variable words as follows. Let $(s_n(x))_{n=0}^
\infty  \in \V$.
 The  constant span of $(s_n(x))_{n=0}^ \infty$ is the
set
$$<(s_n(x))_{n=0}^ \infty >_c=\{s_{l_0}(a_0)... s_{l_n}(a_n):
n\in \nn,\ 0\leq l_0<...<l_n,\; a_0, ... , a_n\in A\}$$
 and the  variable span of $(s_n(x))_{n=0}^ \infty$,
  denoted by $<(s_n(x))_{n=0}^ \infty>_v$,
 is the set
\[ V(A) \cap
\big\{s_{l_0}(a_0)... s_{l_n}(a_n): n\in \nn,\ 0\leq
l_0<...<l_n,\;
 a_0, ... , a_n\in A\cup \{x\}\big\}.\] In the following  we
will also write $<\vs>_c$ (resp. $<\vs >_v$) to denote the constant
(resp. variable) span of an $\vs\in V^{\leq \infty}(A)$.

\subsubsection{Extracted subsequences  of a sequence of variable words}
We start with the following definition.

\begin{defn}\label{ublock} Let $\vs=(s_n(x))_{n=0}^ \infty\in \V$.

(a) Let $l\in \nn$ and  $\vt=(t_n(x))_{n=0}^l\in V^{<\infty}(A)$.
 We say that $\vt$ is a (finite)
extracted subsequence of $\vs$ if there exist $0=m_0<...<m_{l+1}$
such that
$$t_i(x)\in <(s_n(x))_{n=m_{i}}^{m_{i+1}-1} >_v,$$ for all  $0 \leq i \leq l$.

(b) Let  $\vt=(t_n(x))_{n=0}^\infty$. We say that $\vt$ is an
(infinite) extracted  subsequence of $\vs$ if for every $l\in \nn$
the sequence $(t_n(x))_{n=0}^l$ is a finite extracted   subsequence
of $\vs$.
\end{defn}
In the following we will write  $\vt \leq \vs$ whenever $\vt \in
V^{\leq \infty}(A)$, $\vs\in \V$ and $\vt$ is an extracted
subsequence of $\vs$. The next fact follows easily from the above
definitions.

\begin{fact}\label{Cf7}
 Let $\vs  \in \V$.
 \begin{enumerate}
   \item[(i)]  If  $\vt  \in V^{\leq \infty}(A)$ with  $\vt \leq
\vs$ then   $<\vt>_c \subseteq <\vs >_c$ and $<\vt >_v
  \subseteq <\vs>_v.$
  \item[(ii)]  If  $ \vw  \in V^{\leq \infty}(A)$ and
$\vt \in  \V$ with
  $\vw \leq  \vt\leq \vs$ then  $\vw \leq \vs.$
  \end{enumerate}
\end{fact}

\subsubsection{The notion of large families}
The next definition is crucial for the proof of Theorem
\ref{thm3f}.
\begin{defn}\label{defnition__large}
  Let $E \subseteq W(A)$  and $\vs\in\V$.
  Then $E$ will be called
  \textit{large} in $\vs$ if $E \cap <\vw>_c \neq \emptyset,$
  for every infinite extracted subsequence
  $\vw$ of $\vs$.\end{defn}

  We close this subsection with some properties of large
  families.

\begin{fact}\label{Cf13}
   Let $E \subseteq W(A)$ and
   $\vs=(s_n(x))_{n=0}^ \infty \in \V$
   such that $E$ is large in $\vs$. Then
 for every $\vt\leq \vs$ we have that $E$ is large in $\vt$.
\end{fact}

Moreover, arguing by contradiction, we obtain the following.

\begin{fact}\label{Cl14}
   Let $E \subseteq W(A)$ and
  $\vec{s} \in \V$
   such that $E$ is large in $\vs$. Let $r\geq 2$   and let
   $E=\bigcup_{i=1}^r
   E_i$.
    Then there exist $1\leq i \leq r$ and  $\vt\leq \vs$
    such that $E_i$ is large in $\vt$.
\end{fact}
For the following fact, we will need the next definition.

\begin{defn}\label{dfnition}
Let $m \in \nn$ and $(s_n(x))_{n=0}^m \in V^{< \infty} (A)$. We set
$$[(s_n(x))_{n=0}^m ]_c= \big\{s_0(a_0)...s_m(a_m):
 a_0, ... , a_m\in A \big\}$$ and
$$[(s_n(x))_{n=0}^m]_v=V(A) \cap \big\{s_0(a_0)...s_m(a_m):
 a_0, ... , a_m\in A \cup \{x\} \big\}.$$
\end{defn}

The next fact is a direct application of the Hales--Jewett
theorem. In a sense, it is a strengthening of Definition \ref{defnition__large}.
More precisely, we have the following.

\begin{fact}\label{uncCl31}
Let $E \subseteq W(A)$ and $\vs=(s_n(x))_{n=0}^ \infty \in \V$
such that $E$ is large in $\vs$. Then there exist $m \in \nn$ and
$w(x) \in <(s_n(x))_{n=0}^m >_v$ such that $\{w(a): a\in
A\}\subseteq  E$.
\end{fact}

\begin{proof}
Assume to the contrary that the conclusion fails. By induction we
construct
  a sequence $\vw = (w_n(x))_{n=0}^ \infty \leq  \vs$ such
   that $<\vw>_c \subseteq E^c$
   which is a contradiction
   since $E$ is large in
   $\vs$. The general inductive step of the construction is as follows.
   Let $n \geq 1$ and assume that $(w_i(x))_{i=0}^{n-1} \leq  \vs$  and
   $<(w_i(x))_{i=0}^{n-1} >_c \subseteq E^c.$
Let $n_0 \geq 1$ be the least integer  satisfying $w_0(x)  ...
w_{n-1}(x) \in
   <(s_i(x))_{i=0}^{n_0-1}>_v$. We set
     $$N= HJ\big(|A|,2^{(|A|+1)^n}\big).$$
   To each  $w\in   [(s_{n_0+i}(x))_{i=0}^{N-1}]_c$
   we assign the set   $\big\{uw :
   u \in  <(w_i(x))_{i=0}^{n-1} >_c\cup \{\emptyset \} \big\}. $
It is easy to see that $|<(w_i(x))_{i=0}^{n-1} >_c\cup \{\emptyset
\}| \leq (|A|+1)^n.$ Therefore,
 since  either $ uw \in E$ or $uw\in E^c$,
 the above correspondence induces  a $2^{(|A|+1)^n}$-coloring
  of the set $ [(s_{n_0+i}(x))_{i=0}^{N-1}]_c.$
By  the choice of $N$, there exists
   a variable word $w(x) \in  [(s_{n_0+i}(x))_{i=0}^{N-1}]_v$
   such that for each
   $u \in
<(w_i(x))_{i=0}^{n-1}>_c \cup \{\emptyset \}$,
   the set $\{ uw(a): a \in
A\}$  either is included in $E$ or disjoint from $E$. By our
assumption, there is no $u \in <(w_i(x))_{i=0}^{n-1}>_c \cup
\{\emptyset \}$ satisfying the first alternative. So  setting
$w_n(x)=w(x)$ we easily see that
    $(w_i(x))_{i=0}^{n} \leq  \vs$
   and $<(w_i(x))_{i=0}^{n}>_c \subseteq
   E^c.$ The inductive step of the construction of $\vw$  is
   complete and as we have already mentioned in the beginning of
   the proof this leads to a contradiction.
\end{proof}

\subsection{The main arguments}

We pass now to the core of the proof.
 We will need  the next
definition.
\begin{defn}\label{dfnition_11} Let  $E$ and
$ F$ be non empty  subsets of  $W(A)$. We define
  \[E_F=\{z\in W(A):w z \in E \text{ for every }w\in
  F\}.\]
\end{defn}

\begin{lem}\label{Cl16}
   Let $E \subseteq W(A)$ and $\vs=(s_n(x))_{n=0}^ \infty \in \V$
    such that $E$ is large in $\vs$. Then there exist
    $ m\geq1$,
    $w(x) \in <(s_n(x))_{n=0}^{m-1}>_v$
   and $\vt \in \V$ with $\vt \leq (s_{n}(x))_{n=m}^\infty$ such that
    if we set $F=\{w(a):a \in A\}$ then
   $E\cap E_F$ is large in $\vt$.
\end{lem}

\begin{proof}
   We start with the following claim.
\medskip

 \noindent\textit{Claim 1.}
    There exists $n_0 \in \nn$ such that for every  $z\in  <(s_i(x))_{i=n_0+1}^{\infty}>_c$
       there exists $w(x)$ in $<(s_i(x))_{i=0}^{n_0} >_v$    such that $\{w(a) z: a\in A \} \subseteq E.$

 \begin{proof}[Proof of Claim 1]  Assume that the claim is not true.
 Then for every $n\in \nn$ there exists $z\in  <(s_i(x))_{i=n+1}^{\infty}>_c$
 such that for every $w(x)\in <(s_i(x))_{i=0}^{n} >_v$, the set $\{w(a) z: a\in A \}$
  is not contained in  $E.$ Using this assumption  we easily  find  a strictly increasing
  sequence $(k_n)_{n=0}^\infty$ in $\nn$ with $k_0=0$ and a sequence $(z_n)_{n=0}^\infty$
   in $W(A)$ such
   the following are
  satisfied.
  \begin{enumerate}
  \item[(i)] For every $n\in\nn$, we have $z_n \in < (s_i(x))_{i=k_n+1}^{k_{n+1}-1}>_c$.

  \item[(ii)] For every $n\in\nn$ and every variable word $u(x)\in<(s_i(x))_{i=k_0}^{k_{n}}>_v$
   we have that $\{u(a)z_n:a\in A\} \nsubseteq E.$
  \end{enumerate}
 For every $n\in\nn$, we set $v_n(x)=s_{k_n}(x) z_n.$

 By (i) we get that $(s_{k_0}(x)z_0,...,s_{k_n}(x)z_n)\leq \vs,$
   for all $n\in\nn$ and therefore,   $(v_n(x))_{n=0}^\infty \leq
\vs $. Moreover, notice that  every $w(x) \in<(v_n(x))_{n=0}^\infty
>_v$ is of the form $w(x)=u(x)z_n$, for some unique $n\in\nn$ and $$u(x)\in<(v_0(x),...,v_{k_{n-1}}(x), s_{k_n}(x)
>_v.$$ Hence, since $<(v_0(x),...,v_{k_{n-1}}(x), s_{k_n}(x)
>_v\subseteq<(s_i(x))_{i=k_0}^{k_{n}}>_v$,
 by (ii) we get that   $\{w(a):a \in A\} \nsubseteq E,$
 for every $w(x) \in<(v_n(x))_{n=0}^\infty
>_v$.   But since
   $(v_n(x))_{n=0}^\infty  \leq  \vs $,  $E$ is large in
   $(v_n(x))_{n=0}^\infty $ and so by  Fact \ref{uncCl31} we arrive to a contradiction.
\end{proof}
We set $m=n_0+1$.  Also, let $L=<(s_n(x))_{n=0}^{m-1}>_v $ and  for
every $ w(x) \in L$, let $F(w(x)) =\{ w(a): a \in A\}.$

By Claim 1, we have that $<(s_i(x))_{i=m}^{\infty}>_c\subseteq
\bigcup_{w(x)\in L}E_{F(w(x))}$ and therefore, $$E\cap
<(s_i(x))_{i=m}^{\infty}>_c \subseteq \bigcup_{w(x)\in L}E\cap E_{F(w(x))}.$$
Hence, $\bigcup_{w(x)\in L}E\cap E_{F(w(x))}$ is large in
$(s_i(x))_{i=m}^{\infty}$. So,  by Fact \ref{Cl14}, there exist
$w(x)\in L$ and $\vt \leq (s_n(x))_{n=m}^\infty$ such that
 $E\cap E_{F(w(x))}$ is
large in $\vt$, as desired.
\end{proof}

\begin{lem}\label{Cl20}  Let $E \subseteq W(A)$ and $\vs\in \V$ such that
   $E$ is large in $\vs$. Then there exists a sequence
   $(w_n(x))_{n=0}^ \infty \leq \vs $ such that setting
$F_n=< (w_i(x))_{i=0}^n>_c$ we have that
 $E\cap E_{F_n}$ is
large in $(w_{i}(x))_{i=n+1}^\infty$, for all $n \in \nn$.
\end{lem}

\begin{proof}
Iterating Lemma \ref{Cl16} we obtain
   a sequence $(w_n(x))_{n=0}^ \infty$ of variable words,  a
sequence $(\vs_n)_{n=0}^ \infty$ in  $ \V$ with $\vs_0 =\vs$ and a
sequence  $(m_n)_{n=0}^ \infty$ in $ \nn $, with $m_0=0$ and
$m_{n}\geq 1$ for all $n \geq 1$, such that setting
$\vs_n=(s_i^{(n)}(x))_{i=0}^\infty$ then  the following  are
satisfied.
\begin{enumerate}

\item[(i)]
 $w_{n}(x) \in <(s^{(n)}_i(x))_{i=0}^{m_{n+1}-1}>_v$.

\item[(ii)] $\vs _{n+1}$ is an extracted  subsequence of
 $(s_{i}^{(n)}(x))_{i=m_{n+1}}^\infty$.

\item[(iii)]  The set $E\cap E_{F_{n}}$ is large in $\vs_{n+1}$.
\end{enumerate}
The above construction is straightforward; we only mention that
for the proof of (iii) we use the following identity.
$$(E \cap E _{F_{n}} ) \cap (E \cap E _{F_{n}} )_{<w_{n+1}(x)>_c}
= E \cap E _{F_{n+1}}.$$ By conditions (i) and (ii) we easily see
that $\vw\leq \vs$ and moreover, for every $n \in \nn$,
$(w_{i}(x))_{i=n}^\infty \leq \vs _n $.  Hence by condition (iii)
 and Fact
 \ref{Cf13} we obtain that $E\cap E_{F_{n}}$ is large in
$(w_{i}(x))_{i=n+1}^\infty$, for all $n \in \nn$.
\end{proof}

\begin{cor}\label{corol22}
   Let  $E \subseteq W(A)$ and $\vs \in \V$ such that
   $E$ is large in $\vs$.
   Then there exists an extracted
    subsequence $\vec{t}=(t_n(x))_{n=0}^\infty$ of $\vs$
   such that
    $<\vec{t}>_c \subseteq E.$
\end{cor}

\begin{proof}
 Let $\vec{w}=(w_n(x))_{n=0}^ \infty$
  be the sequence obtained in Lemma
 \ref{Cl20}. Since $E$ is large in $\vs$ and
 $\vw \leq \vs$, we obtain that $E $ is large in $\vw$ and
 therefore, by Fact \ref{uncCl31}, we have that there
 exist $r_1 \geq 1$ and $t_{0}(x) \in
 <(w_{i}(x))_{i=0}^{r_1-1}>_v$
 such that $\{ t_0(a): a \in A \}
 \subseteq E$. We set $ G_0= <t_0(x)>_c$ and
 $F= <(w_i(x)_{i=0}^{r_1-1}>_c$. By Lemma \ref{Cl20},
 we have that $E \cap E _F$ is large in $(w_i(x))_{i=r_1}^
 \infty.$ Since $G_0 \subseteq F$ we get that
 $E _F \subseteq E_{G_0}$ and hence, $E \cap E_{G_0}$ is large $(w_i(x))_{i=r_1}^
 \infty$. Using again Fact \ref{uncCl31}, applied for $E \cap E_{G_0}$, we find $r_2 > r_1$ and a variable word
 $t_1(x) \in < (w_i(x))_{i=r_1} ^ {r_2-1}>_v $ such that $<t_1(x)>_c \subseteq E \cap
 E_{G_0}$. We set $G_1=<(t_0(x),t_1(x))>_c$ and we notice that
 $G_1 \subseteq E$ and also that
 $E\cap E_{G_1}$ is large in $(w_i(x))_{i=r_2}^\infty$.
Proceeding similarly, we construct a sequence $ \vt
=(t_n(x))_{n=0}^\infty \leq \vw \leq \vs $ such that
$<(t_i(x))_{i=0}^{n}>_c \subseteq E$, for all $n \in \nn$.
Hence, $\vt \leq \vs$ and $<\vt>_c \subseteq E$, as desired.
\end{proof}

\begin{proof}[\textbf{Proof of Theorem \ref{thm3f}}] Let $r \geq 2 $ and
let $W(A) = \cup _{i=1}^r E_i$. Let $\vec{v}=(x,x,...)$. Then
 $<\vec{v}>_c = W(A)=\cup_{i=1}^r E_i$.
Trivially,  $\cup_{i=1}^r E_i$ is
  large in
  $\vec{v}$.
Hence, by Fact \ref{Cl14}, there
  exist $1 \leq i \leq r$ and $\vs \leq \vec{v} $
  such that $E_i$ is large in $\vs$. By Corollary \ref{corol22},
  there exists
   $\vec{t} \leq  \vs$
   such that
    $<\vec{t}>_c \subseteq E_i$ and the proof is complete.
\end{proof}

\section{Proof of Theorem \ref{thm3}}

In this section we present the proof of Theorem \ref{thm3}. As
mentioned in the Introduction, the arguments are similar to those
of Theorem \ref{thm3f}. The main difficulty that we encountered
was the manipulation of the infinite sequence of finite alphabets
which, among others, it increases the complexity of the notation.
We start by reformulating the basic terminology from Section
\ref{main proof}, by taking into consideration the infinite
sequence of alphabets.

\subsection{Preliminaries}
We fix  an increasing  sequence
 $$A_0\subseteq A_1 \subseteq ... \subseteq A_n \subseteq ...$$
 of finite alphabets and we set $$A=\bigcup_{n\in \nn} A_n.$$
 Let
 $V(A)$ be the set of all \textit{variable words} (\textit{over}
 $A$).
 By $V^ {<\infty}(A) $ (resp. $V^ \infty(A)
$) we denote the set of all finite (resp. infinite) sequences of
variable words. Also let $V^{\leq\infty}(A)=V^ {<\infty}(A) \cup
V^ \infty(A)$.

\subsubsection{Constant and variable span of a sequence of variable words with respect to a sequence
of finite alphabets}
 Let  $m\in \nn$,
$(s_n(x))_{n=0}^m$ in $V^{<\infty}(A)$ and $(k_n)_{n=0}^m$ be a
strictly increasing finite sequence of non negative integers.
 The \textit{constant span of}
$(s_n(x))_{n=0}^m$ \textit{with respect to} $(A_{k_n})_{n=0}^m$
denoted by $<(s_n(x))_{n=0}^m \  \| \ (A_{k_n})_{n=0}^m>_c$
is defined to be the set
$$\bigcup_{n=0}^m\{s_{l_0}(a_0)... s_{l_n}(a_n):
0 \leq l_0<...<l_n\leq m,\;
  a_i\in A_{k_{l_i}},
\  0\leq i \leq n \}.$$ We also define the \textit{variable span of}
$(s_n(x))_{n=0}^m$ \textit{with respect to} $(A_{k_n})_{n=0}^m$,
denoted by $<(s_n(x))_{n=0}^m \  \| \ (A_{k_n})_{n=0}^m>_v$, to be the set
\[V(A) \cap\bigcup_{n=0}^m\{s_{l_0}(a_0)... s_{l_n}(a_n):
0 \leq l_0<...<l_n\leq m,
  a_i\in A_{k_{l_i}}\cup\{x\},
 0\leq i \leq n \}.\] The above notation is extended  to
infinite sequences of variable words as follows. Let
$(s_n(x))_{n=0}^ \infty \in \V$ and $(k_n)_{n=0}^ \infty$ be a
strictly increasing sequence of non negative integers. Then the
constant span of $(s_n(x))_{n=0}^ \infty$ with respect to
$(A_{k_n})_{n=0}^ \infty$ denoted by $<(s_n(x))_{n=0}^ \infty \  \|
\ (A_{k_n})_{n=0}^ \infty>_c$ is defined to be the set
$$\{s_{l_0}(a_0)... s_{l_n}(a_n):
n\in \nn,\ 0\leq l_0<...<l_n,\; a_i\in A_{k_{l_i}}, 0\leq i\leq n
\}.$$ Similarly,  the  variable span of $(s_n(x))_{n=0}^ \infty$
with respect to $(A_{k_n})_{n=0}^ \infty$, denoted by
$<(s_n(x))_{n=0}^ \infty \  \| \ (A_{k_n})_{n=0}^ \infty>_v$, is the
set
\[
V(A) \cap \{s_{l_0}(a_0)... s_{l_n}(a_n): n\in \nn,0\leq
l_0<...<l_n, a_i\in A_{k_{l_i}}\cup\{x\}, 0\leq
i\leq n \}.\] In the following  we also write $<\vs \  \| \
(A_{k_n})_{n=0}^ \infty>_c$ (resp. $<\vs \  \| \ (A_{k_n})_{n=0}^
\infty>_v$) to denote the the  constant (resp. variable) span of
$\vs=(s_n(x))_{n=0}^ \infty$ with respect to $(A_{k_n})_{n=0}^
\infty$.

\subsubsection{Extracted $k$-subsequences  of a sequence of variable words}
In this subsection, we extend the notion of extracted subsequences defined
in Section \ref{main proof}.
\begin{defn}\label{iublock}
Let $k \in \nn$ and $\vs=(s_n(x))_{n=0}^ \infty\in \V$.

(a) Let $l\in \nn$ and
   $\vt=(t_n(x))_{n=0}^l\in V^{<\infty}(A)$. We say that
 $\vt$ is a
(finite) extracted $k$-subsequence of $\vs$ if there exist
$0=m_0<...<m_{l+1}$ such that
$$t_i(x)\in <(s_n(x))_{n=m_{i}}^{m_{i+1}-1} \ \| \
(A_{k+n})_{n=m_{i}}^{m_{i+1}-1}>_v,$$ for all $0 \leq i \leq l$.

(b) Let $\vt=(t_n(x))_{n=0}^\infty \in \V$. We say that $\vt$ is a
(infinite) extracted  $k$-subsequence of $\vs$ if for every $l\in
\nn$, the sequence $(t_n(x))_{n=0}^l$ is a finite extracted
$k$-subsequence of $\vs$.
\end{defn}

In the following we will write  $\vt \leq_k \vs$, whenever $\vt \in
V^{\leq \infty}(A)$, $\vs\in \V$ and $\vt$ is an extracted
$k$-subsequence of $\vs$. Taking into account that  the sequence of
alphabets $(A_n)_{n=0}^ \infty$ is increasing, the next fact follows
easily from the
 above definitions.
\begin{fact}\label{Cf25}
 Let $k, l  \in \nn$ and $\vs,\vt,\vw  \in \V$.
 \begin{enumerate}
   \item[(i)]  If  $\vt \leq _k
\vs$ then   $$<\vt \; \| \; (A_{k+n})_{n=0}^ \infty>_c
\subseteq <\vs\;\|\;
  (A_{k+n})_{n=0}^ \infty>_c$$ and $$<\vt \; \| \;
  (A_{k+n})_{n=0}^ \infty>_v \subseteq <\vs\;\|\;
(A_{k+n})_{n=0}^ \infty>_v.$$

\item[(ii)] If
 $k\leq l $ and  $\vw \leq _{k} \vt\leq_{l} \vs$ then
 $\vw \leq_{l} \vs.$
 \end{enumerate}
\end{fact}

\subsubsection{The notion of $k$-large families}
The following is an extension of Definition \ref{defnition__large}.

\begin{defn}\label{udefn large}
  Let $k \in \nn$, $E \subseteq W(A)$  and $\vs\in\V$.
  Then $E$ will be called
  $k$-\textit{large} in $\vs$ if $E \cap <\vw  \  \| \
  (A_{k+n})_{n=0}^ \infty>_c \neq \emptyset,$
  for every infinite extracted $k$-subsequence
  $\vw$ of $\vs$.\end{defn}

 We close this subsection with some properties of $k$-large
  families.

\begin{fact}\label{ublock fact}
  Let $k \in \nn$, $E \subseteq W(A)$ and $\vs=(s_n(x))_{n=0}^ \infty \in \V$
such that $E$ is $k$-large in $\vs$. Then the following hold true.
\begin{enumerate}

\item[(i)]  $E$ is $k$-large in $\vt$,  for every infinite
extracted $k$-subsequence $\vt$ of  $\vs$.

  \item[(ii)] For every $m \in \nn$,
$E$ is $(k+m)$-large in $(s_n(x)_{n=m}^\infty$.

\end{enumerate}
\end{fact}

\begin{proof}
  (i) It follows easily using the first part of Fact \ref{Cf25}.

  (ii) Let $\vt \in \V$ such that $\vt \leq_ {k+m} (s_n(x))_{n=m}^\infty$.
  It is easy to check that  $\vt\leq_k\vs$ and therefore,
  since $E$ is $k$-large in $\vs$, we obtain that
  $$E \cap <\vt  \  \| \ (A_{k+n})_{n=0}^ \infty>_c \neq
  \emptyset.$$ Moreover, since the
  sequence of alphabets $(A_n)_{n=0}^\infty$ is increasing we get that
  $$<\vt  \  \| \ (A_{k+n})_{n=0}^ \infty>_c
  \subseteq <\vt  \  \| \
  (A_{k+m+n})_{n=0}^ \infty>_c.$$
  Hence, $E\cap <\vt  \  \| \
  (A_{k+m+n})_{n=0}^ \infty>_c \neq \emptyset$ for every
  $\vt \leq _{k+m}
  (s_n(x))_{n=m}^\infty$, i.e. $E$ is $(k+m)$-large in
  $(s_n(x)_{n=m}^\infty$.
\end{proof}

\begin{fact}\label{Cl30}
   Let $k \in \nn$, $E \subseteq W(A)$ and
  $\vec{s} \in \V$
   such that $E$ is $k$-large in $\vs$. Let $r\geq 2$  and
   let $E=\bigcup_{i=1}^r
   E_i$. Then there exist $1\leq i \leq r$ and an
   infinite extracted $k$-subsequence $\vt$
   of $\vs$ such that $E_i$ is $k$-large in $\vt$.
\end{fact}

\begin{defn}\label{definition reduced finite}
Let $m \in \nn$ and $(s_n(x))_{n=0}^m \in V^{< \infty} (A)$ and let
$B$ be a finite subset of $A$. We set  $$[(s_n(x))_{n=0}^m  \; \| \;
B ]_c= \big\{s_0(a_0)...s_m(a_m):
 a_0, ... , a_m\in B \big\}$$ and
$$[(s_n(x))_{n=0}^m \; \| \;B]_v=V(A) \cap \big\{s_0(a_0)...s_m(a_m):
 a_0, ... , a_m\in B \cup \{x\} \big\}.$$
\end{defn}

\begin{fact}\label{Cl31}
Let $k \in \nn$, $E \subseteq W(A)$ and $\vs=(s_n(x))_{n=0}^ \infty \in \V$
such that $E$ is $k$-large in $\vs$. Then there exist $m \in \nn$
and $w(x) \in <(s_n(x))_{n=0}^m \  \| \ (A_{k+n})_{n=0}^m>_v$ such
that $\{w(a): a\in A_k\}\subseteq  E$.
\end{fact}

\begin{proof}
Assume to the contrary that the conclusion fails. By induction we
construct
  a sequence $\vw = (w_n(x))_{n=0}^\infty  \leq _k \vs$ such
   that $<\vw \; \| \; (A_{k+n})_{n=0}^ \infty>_c \subseteq E^c$
   which is a contradiction
   since $E$ is $k$-large in
   $\vs$.
The general inductive step of the construction is as follows.
   Let $n \geq 1$ and assume that
   $ (w_i(x))_{i=0}^{n-1} \leq _k \vs$  and
   $<(w_i(x))_{i=0}^{n-1} \; \| \;
   (A_{k+i})_{i=0}^{n-1}>_c \subseteq
   E^c.$
Let $n_0 \geq 1$ be the least integer  satisfying $$w_0(x)  ...
w_{n-1}(x) \in
   <(s_i(x))_{i=0}^{n_0-1}\;\|\; (A_{k+i})_{i=0}^{n_0-1}>_v$$
   and let  $$N= HJ(|A_{k+n}|,2^{ \prod_{i=0}^{n-1}(|A_{k+i}|+1)}).$$
  To each  $w \in   [(s_{n_0+i}(x))_{i=0}^{N-1} \  \|
\ A_{k+n}]_c$
   we assign the set of words
    $$\big\{uw :
   u\in   <(w_i(x))_{i=0}^{n-1} \; \| \;
   (A_{k+i})_{i=0}^{n-1}>_c\cup \{\emptyset \} \big\} . $$
Since $ uw \in E$ or $E^c$,
  it is easy to see that the above correspondence induces  a $2^{ \prod_{i=0}^{n-1}(|A_{k+i}|+1)}$-coloring
  of
    the set $ [(s_{n_0+i}(x))_{i=0}^{N-1} \  \| \ A_{k+n}]_c.$ Hence,
   by the Hales--Jewett theorem and the choice of $N$, there exists
   a variable word $$w(x) \in [(s_{n_0+i}(x))_{i=0}^{N-1} \  \| \
   A_{k+n}]_v$$
   such that for each
   $u \in
 <(w_i(x))_{i=0}^{n-1} \; \| \;
   (A_{k+i})_{i=0}^{n-1}>_c\cup \{\emptyset \} $  the set
   $\{ u w(a): a \in
A_{k+n}\}$  either is included in $E$ or is disjoint from $E$. By
our initial assumption, there is no $u \in  <(w_i(x))_{i=0}^{n-1} \;
\| \;
   (A_{k+i})_{i=0}^{n-1}>_c\cup \{\emptyset \}$
    satisfying the first alternative. Setting
 $w_n(x)=w(x)$ we easily see that
    $(w_i(x))_{i=0}^{n} \leq _k \vs$
   and $<(w_i(x))_{i=0}^{n} \; \| \; (A_{k+i})_{i=0}^{n}>_c \subseteq
   E^c.$ The inductive step of the construction of $\vw$  is
   complete and as we have already mentioned in the beginning of
   the proof this leads to a contradiction.
\end{proof}

\subsection{The main arguments}

The next lemma corresponds to Lemma \ref{Cl16} and constitutes
the core of the proof of Theorem \ref{thm3}.

Recall that for every non empty  subsets $E$, $ F$  of $W(A)$ we
have  set
  \[E_F=\{z\in W(A):w z \in E \text{ for every }w\in
  F\}.\]
\begin{lem}\label{Cl32}
   Let $k \in \nn$, $E \subseteq W(A)$ and $\vs=(s_n(x))_{n=0}^ \infty \in \V$
    such that $E$ is $k$-large in $\vs$. Then there exist
    $ m\geq1$, a variable word
    $$w(x)\in <(s_n(x))_{n=0}^{m-1}\ \| \   (A_{k+n})_{n=0}^{m-1}>_v$$
   and $\vt \in \V$ with $\vt \leq _{k+m} (s_{n}(x))_{n=m}^\infty$ such that
    setting $F=\{w(a):a \in A_{k}\}$ then
   $E\cap E_F$ is $(k+m)$-large in $\vt$.
\end{lem}

\begin{proof}

  We start with the following claim.
\medskip

 \noindent\textit{Claim 1.}
    There exists $n_0 \in \nn$ such that for every
    $$z\in  <(s_i(x))_{i=n_0+1}^{\infty} \  \| \
    (A_{k+i})_{i=n_0+1}^\infty>_c$$
       there exists $w(x)$ in $<(s_i(x))_{i=0}^{n_0}  \  \| \ (A_{k+i})_{i=0}^{n_0}>_v$
       such that $$\{w(a) z: a\in A_k \} \subseteq E.$$

 \begin{proof}[Proof of Claim 1]  Assume that the claim is not true.
 Then for every $n\in \nn$ there exists
 $z\in  <(s_i(x))_{i=n+1}^{\infty} \  \| \ (A_{k+i})_{i=n+1}^\infty>_c$
 such that
the set
 $\{w(a) z: a\in A_k \}$
  is not contained in  $E$,  for every $w(x)\in <(s_i(x))_{i=0}^{n_0}  \  \| \ (A_{k+i})_{i=0}^{n_0}>_v$. Using this assumption  we easily  find  a strictly increasing
  sequence $(k_n)_{n=0}^\infty$ in $\nn$ with $k_0=0$ and a sequence $(z_n)_{n=0}^\infty$
   in $W(A)$ such the following are satisfied.
  \begin{enumerate}
  \item[(i)] For every $n\in\nn$, we have
  $z_n \in < (s_i(x))_{i=k_n+1}^{k_{n+1}-1}  \  \| \ (A_{k+i})_{i=k_n+1}^{k_{n+1}-1}>_c$.

  \item[(ii)] For every $n\in\nn$ and every
  $u(x)\in<(s_i(x))_{i=k_0}^{k_{n}}  \  \| \ (A_{k+i})_{i=k_0}^{k_{n}}>_v$
   we have that $\{u(a)z_n:a\in A_k\} \nsubseteq E.$
  \end{enumerate}
We set $v_n(x)=s_{k_n}(x) z_n$, for every $n \in \nn$.

  By (i), we have that $(s_{k_0}(x)z_0,...,s_{k_n}(x)z_n)\leq_k \vs$,
   for every $n\in\nn$ and therefore,   $(v_n(x))_{n=0}^\infty \leq_k \vs $.
Moreover, since the sequence $(k_n)_{n=0}^\infty$ is
strictly increasing and the sequence of finite alphabets
 $(A_n)_{n=0}^\infty$ is increasing,
 by (ii), we obtain that
 $\{u(a)z_n:a\in A_k\} \nsubseteq E,$
  for every $n\in\nn$ and every
 $$u(x)\in<(s_{k_0}(x)z_0,  ..., s_{k_{n-1}}(x)z_{n-1}, s_{k_n}(x)
  \  \| \ (A_{k+i})_{i=0}^n>_v.$$
Hence, since  every
  $w(x) \in<(v_n(x))_{n=0}^\infty \  \| \ (A_{k+n})_{n=0}^\infty>_v$
is of the form $w(x)=u(x)z_n$ for some unique $n\in\nn$ and  some
variable word $$u(x)<(s_{k_0}(x)z_0,  ...,
s_{k_{n-1}}(x)z_{n-1}, s_{k_n}(x) \  \| \ (A_{k+i})_{i=0}^n>_v,$$ we
conclude  that there is no $w(x)  \in <(v_n(x))_{n=0}^\infty \ \| \
(A_{k+n})_{n=0}^\infty>_v$ such that $\{w(a):a \in A_k\} \subseteq
E.$ But since
   $(v_n(x))_{n=0}^\infty  \leq_k  \vs $, we have that  $E$ is $k$-large in
   $(v_n(x))_{n=0}^\infty $ and so by  Fact \ref{Cl31} we arrive to a contradiction.
\end{proof}

  We set $m=n_0+1$.  Also, let $L=<(s_n(x))_{n=0}^{m-1} \  \| \ (A_{k+i})_{i=0}^{m-1}>_v $
  and  for
every $ w(x) \in L$, let $F(w(x)) =\{ w(a): a \in A_k\}.$ By Claim
1, we have that $<(s_i(x))_{i=m}^{\infty} \  \| \
(A_{k+i})_{i=m}^\infty >_c\subseteq \bigcup_{w(x)\in L}E_{F(w(x))}$
and therefore, $$E\cap <(s_i(x))_{i=m}^{\infty}\  \| \
(A_{k+i})_{i=m}^\infty>_c   \subseteq \bigcup_{w(x)\in L}E\cap E_{F(w(x))}.$$
By part (ii) of Fact \ref{ublock fact}, we have that $E$ is
$(k+m)$-large in $(s_i(x))_{i=m}^{\infty}$. Hence, $\bigcup_{w(x)\in
L}E\cap E_{F(w(x))}$ is $(k+m)$-large in $(s_i(x))_{i=m}^{\infty}$.
So, by Fact \ref{Cl30} there exist a variable word $w(x)\in L$ and
$\vt \leq_{k+m} (s_n(x))_{n=m}^\infty$ such that
 $E\cap E_{F(w(x))}$ is $(k+m)$-large in $\vt$ and the proof is complete.
  \end{proof}

\begin{lem}\label{unccl18}  Let $k \in \nn$, $E \subseteq W(A)$ and
 $\vs= (s_n(x))_{n=0}^ \infty \in \V$ such that
   $E$ is $k$-large in $\vs$. Then there exist a sequence
   $\vw =(w_n(x))_{n=0}^ \infty \leq _k \vs $ and two
strictly increasing sequences $(k_n)_{n=0}^ \infty$ and
$(p_n)_{n=0}^ \infty$ in $ \nn $, with $k_0=k$ and $p_0=0$, such
that for every $n \in \nn$, the following  are satisfied.
\begin{enumerate}
 \item[(W1)] $k+p_n \geq k_n$

 \item[(W2)]  $w_n(x) \in
<(s_i(x))_{i=p_n}^{p_{n+1}-1} \; \| \;
(A_{k+i})_{i=p_n}^{p_{n+1}-1}>_v$. \item[(W3)] Setting  $F_n=<
(w_i(x))_{i=0}^n\;\|\; (A_{k_i})_{i=0}^{n}>_c$ then   $E\cap
E_{F_n}$ is $k_{n+1}$-large in $(w_{i}(x))_{i=n+1}^\infty$.
\end{enumerate}
\end{lem}

\begin{proof} We start with the following.
\medskip

 \noindent\textit{Step 1.}  Let $k \in \nn$, $E \subseteq W(A)$ and
 $\vs= (s_n(x))_{n=0}^ \infty \in \V$ such that
   $E$ is $k$-large in $\vs$. Then there exist (a) a sequence $(w_n(x))_{n=0}^ \infty$ of variable
words, (b)  a sequence $(\vs_n)_{n=0}^ \infty$ in  $ \V$ with $\vs_0
  =\vs$,  (c) two sequences $(m_n)_{n=0}^ \infty$ and $(k_n)_{n=0}^\infty$ in
$ \nn $, with  $m_0=0$ and $k_0=k$
 such that setting for every $n \in \nn$,
$\vs _n = (s_i ^{(n)}(x))_{i=0}^\infty $ then the following are
satisfied.
\begin{enumerate}
\item[(i)] $m_{n+1}\geq 1$ and
$k_{n+1}=k_{n}+m_{n+1}$.

\item[(ii)]
 $w_{n}(x) \in <(s^{(n)}_i(x))_{i=0}^{m_{n+1}-1}\ \| \
(A_{k_{n}+i})_{i=0}^{m_{n+1}-1}>_v$.

\item[(iii)]  $\vs _{n+1}$ is an extracted
$k_{n+1}$-subsequence of $(s_{i}^{(n)}(x))_{i=m_{n+1}}^\infty$.

\item[(iv)] If we set $ F_n=<
(w_i(x))_{i=0}^n\;\|\; (A_{k_i})_{i=0}^{n}>_c$ then $E\cap
 E_{F_{n}}$ is  $k_{n+1}$-large in  $\vs_{n+1}$.
\end{enumerate}

\begin{proof}[Proof of Step 1] For $n=0$ we set $m_0=0$, $k_0=0$ and $\vs_0=\vs$.
Assume that the construction has been carried out up to some
$n\in\nn$, i.e. the sequences $(w_i(x))_{i<n},$ $(\vs_i)_{i=0}^{n}$,
$(m_i)_{i=0}^{n}$, $(k_i)_{i=0}^{n}$  have been selected. We set
$G=E\cap E_{F_{n-1}}$ (if $n=0$, we set $G=E$). By our inductive
assumptions we have that $G$ is $k_n$-large in $\vs_n$. Therefore,
by Lemma \ref{Cl32}, there exist $ m\geq 1$, a variable word $w(x)
\in <(s^{(n)}_i(x))_{i=0}^{m-1}\ \| \ (A_{k_n+i})_{i=0}^{m-1}>_v$
and an extracted $({k_n+m})$-subsequence $\vt$ of
$(s^{(n)}_i(x))_{i=m}^\infty$, such that   $G\cap G_F$ is
$(k_n+m)$-large in $\vt$, where $F=\{w(a):a \in A_{k_n}\}$. We set
$$m_{n+1} =m, \ k_{n+1} = k_n +m, \ w_n(x) =w(x) \text{ and  }\ \vs_{n+1}=
\vt.$$ Moreover, if $F_n=< (w_i(x))_{i=0}^n\;\|\;
 (A_{k_i})_{i=0}^{n}>_c$ then it is easy to check that  $G\cap G_F=E\cap E_{F_{n}}.$
The above choices clearly fulfill  conditions (i)-(iv) and the proof
of the inductive step of the construction is complete.
\end{proof}

\noindent\textit{Step 2}. Let $(\vs_n)_{n=0}^ \infty$, $(m_n)_{n=0}^
\infty$ and $(k_n)_{n=0}^\infty$ be the sequences obtained in Step
1. Then there exists a strictly increasing sequence
$(p_n)_{n=0}^\infty$ in $\nn$ with $p_0=0$  such that for every $n \in \nn$,
  the following are satisfied.
\begin{enumerate}
\item [(v)] $k+p_{n}\geq k_{n}$. \item [(vi)] The set
$<(s^{(n)}_i(x))_{i=0}^{m_{n+1}-1}\ \| \
(A_{k_{n}+i})_{i=0}^{m_{n+1}-1}>_v$ is a subset of\\
$<(s_i(x))_{i=p_{n}}^{p_{n+1}-1} \; \| \;
(A_{k+i})_{i=p_{n}}^{p_{n+1}-1}>_v$. \item [(vii)] $\vs _{n}$ is
an extracted $(k+p_{n})$-subsequence of
$(s_{i}(x))_{i=p_{n}}^\infty$.
\end{enumerate}
\begin{proof}[Proof of Step 2] We set $p_0=0$ and we easily see that
(v) and (vii) are satisfied for $n=0$.
 Let $n\in\nn$ and assume that the
sequence $(p_i)_{i=0}^{n}$ has been selected.   By (vii) we obtain
that for every $m\geq 1$
 there exists a sequence $(I_j)_{j=0}^{m-1}$ of successive nonempty
intervals of $\nn$ with $\min (I_0)=0$ such that setting
$M(m)=\max(I_{m-1})+1$, then
\begin{equation}\label{eq1}s^{(n)}_j(x)\in
<(s_{p_n+i}(x))_{i\in I_j} \ \|\ (A_{k+p_n+i})_{i\in I_j}>_v\end{equation}
for every $j\in\{0,\dots, m-1\}$ and,
\begin{equation}\label{eq2} (s^{(n)}_i(x))_{i=m}^
 \infty \leq _{k+p_n+M(m)} (s_i(x))_{i=p_n+M(m)}^ \infty.
 \end{equation}
We claim that we may set $$p_{n+1}=p_n+M(m_{n+1}).$$ Indeed, by  our
inductive assumptions we have that $k+p_n\geq k_n$ and therefore,
since the sequence of the alphabets $(A_n)_{n=0}^\infty$
 is increasing, by
\eqref{eq1} (for $m=m_{n+1}$), we conclude that
$$<(s^{(n)}_i(x))_{i=0}^{m_{n+1}-1}\ \| \
(A_{k_{n}+i})_{i=0}^{m_{n+1}-1}>_v\subseteq
<(s_i(x))_{i=p_n}^{p_{n+1}-1} \; \| \;
(A_{k+i})_{i=p_n}^{p_{n+1}-1}>_v$$ and so (vi) is satisfied.
Moreover,  notice that $M(m)\geq m$. Hence,
\begin{equation}k+p_{n+1}=k+p_n+M(m_{n+1})\geq   k+p_n+m_{n+1}\geq k_n+ m_{n+1}=k_{n+1},
\end{equation} that is
(v) is also satisfied.  Finally, by (iii) of Step 1 and \eqref{eq2}
above, we have
$$\vs _{n+1} \leq _{k_{n+1}}
(s_{i}^{(n)}(x))_{i=m_{n+1}}^\infty\leq_{k+p_{n+1}}
(s_i(x))_{i=p_{n+1}}^ \infty.$$ Since $k_{n+1}\leq k+p_{n+1}$, by
part (ii) of Fact \ref{Cf25}, we obtain that $$\vs _{n+1}
\leq_{k+p_{n+1}} (s_i(x))_{i=p_{n+1}}^\infty.$$ Hence, (vii) is also
valid and the inductive step of the construction is complete.
\end{proof}
If  $n_0\geq 1$  then  starting from $\vs_{n_0}$ and $k_{n_0}$
instead of $\vs_0=\vs$ and $k_0=k$ and working as in Step 2  we
derive the following.
\medskip

\noindent\textit{Step 3}. Let   $n_0\geq 1$ and let $(\vs_n)_{n=0}^
\infty$, $(m_n)_{n=0}^ \infty$ and $(k_n)_{n=0}^\infty$ be the
sequences obtained in Step 1. Then there exists a strictly
increasing sequence $(q_n)_{n=0}^\infty$ in $\nn$ with $q_0=0$ such
that
 for every $n\in\nn$
the following are satisfied.
\begin{enumerate}
\item [(v$'$)] $k_{n_0}+q_{n}\geq k_{n_0+n}$. \item [(vi$'$)] The
set $<(s^{(n_0+n)}_i(x))_{i=0}^{m_{n_0+n+1}-1}\ \| \
(A_{k_{n_0+n}+i})_{i=0}^{m_{n_0+n+1}-1}>_v$ is a subset of the
set $ <(s^{(n_0)}_i(x))_{i=q_{n}}^{q_{n+1}-1} \; \| \;
(A_{k_{n_0}+i})_{i=q_{n}}^{q_{n+1}-1}>_v$.

\item [(vii$'$)] $\vs _{n_0+n}$ is an extracted
$(k_{n_0}+q_{n})$-subsequence of
$(s^{(n_0)}_{i}(x))_{i=q_{n}}^\infty$.
\end{enumerate}
\medskip

We are now ready to complete the proof of the lemma. Clearly, condition (W1)
 follows by
(v). Also, by (ii) and (vi) we obtain that
 $$w_{n}(x) \in <(s_i(x))_{i=p_{n}}^{p_{n+1}-1} \; \| \;
(A_{k+i})_{i=p_{n}}^{p_{n+1}-1}>_v,$$ for every $n\in \nn$ and so
(W2) is also satisfied. It remains to verify (W3). To this end, let
$n_0$ be an arbitrary positive integer.  Then, by (ii) and (vi$'$)
we get that
 $$w_{n_0 +n}(x) \in
 <(s^{(n_0)}_i(x))_{i=q_{n}}^{q_{n+1}-1} \; \| \;
(A_{k_{n_0}+i})_{i=q_{n}}^{q_{n+1}-1}>_v,$$ for every $n \in\nn$,
that is $(w_{i}(x))_{i=n_0}^\infty$ is an extracted
 $k_{n_0}$-subsequence
of $\vs_{n_0}$. Therefore, for every $n\in\nn$,
$(w_{i}(x))_{i=n+1}^\infty$ is a
 $k_{n+1}$-subsequence
of $\vs_{n+1}$ and so, by (iv) of Step 1 and part (i) of Fact \ref{ublock fact},
 we get that for every $n \in \nn$,  $E\cap
E_{F_{n}}$ is $k_{n+1}$-large in $(w_{i}(x))_{i=n+1}^\infty$, that is (W3). Finally, setting
$\vw = (w_n(x))_{n=0}^\infty$, by (W2) we obtain that
$\vw \leq _{k} \vs$ and the proof is complete. \end{proof}

\begin{cor}\label{unccl19}
   Let $k \in \nn$, $E \subseteq W(A)$ and $\vs \in \V$ such that
   $E$ is $k$-large in $\vs$.
   Then there exists an infinite extracted
   $k$-subsequence $\vec{t}$ of $\vs$
   such that
    $$<\vec{t} \;\|\; ( A_{k+n})_{n=0}^ \infty>_c \subseteq E.$$
\end{cor}

\begin{proof}
  Let $\vec{w}=(w_n(x))_{n=0}^ \infty$, $(k_n)_{n=0}^ \infty$
 and  $(p_n)_{n=0}^ \infty$ be the sequences obtained in Lemma
 \ref{unccl18}. We start with the following claim.
 \medskip

\noindent\textit{Claim 1.} There exist a strictly increasing
sequence
 $(r_n)_{n=0}^ \infty$ in $\nn$ with $r_0 =0$
 and a sequence $\vt=(t_n(x))_{n=0}^ \infty$  of variable words
 such that
for every  $n\in \nn$ the following are
  satisfied.
 \begin{enumerate}

 \item[(T1)] $t_{n}(x) \in
 <(w_{r_n+i}(x))_{i=0}^{r_{n+1}-r_{n}-1}
  \; \|\; (A_{k_{r_{n}}+i})_{i=0}^{r_{n+1}-r_{n}-1}>_v$.

 \item[(T2)] $<(t_i(x))_{i=0}^n \; \| \;
  (A_{k_{r_{i}}})_{i=0}^n>_c \subseteq E$.

 \end{enumerate}

\begin{proof}[Proof of Claim 1]
Since $ \vw \leq _k \vs$ and $E$ is $k$-large in $\vs$ we get that
$E$ is $k$-large in $\vw$. Therefore, by Fact \ref{Cl31} there
exist a positive integer $r_1$ and a variable word $$t_0(x)\in
<(w_i(x))_{i=0}^{r_1-1 }\; \|\;
  (A_{k+i})_{i=0}^{r_1-1} >_v $$ such that $\{t_0(a):a\in
A_k\}\subseteq E$. Since $k_0=k$ and $r_0=0$ we have that conditions (T1) and (T2) are
satisfied for $n=0$.

We set $G_0=\{t_0(a):a\in A_k\}$. Notice that $G_0\subseteq
<(w_i(x)_{i=0}^{r_1-1}\|(A_{k_i})_{i=0}^{r_1-1}>_c $ and so, by (W3)
of Lemma \ref{unccl18}, we get that $E\cap E_{G_0}$ is
$k_{r_1}$-large in $(w_i(x))_{i=r_1}^\infty$. Hence, again by Fact
\ref{Cl31}, there exists an integer $r_2> r_1$ and a variable word
 $$t_{1}(x) \in
 <(w_{r_1+i}(x))_{i=0}^{r_{2}-r_{1}-1}
  \; \|\; (A_{k_{r_{1}}+i})_{i=0}^{r_{2}-r_{1}-1}>_v$$ such that
  $\{t_1(a):a\in A_{k_{r_{1}}} \} \subseteq E\cap E_{G_0}$.
Observe that conditions (T1) and (T2) are also
satisfied for $n=1$. Continuing in the same way, we select the
desired sequence $\vt$.
\end{proof}

 \noindent\textit{Claim 2.} For every  $n\in \nn$, $t_{n}(x) \in
 <(s_i(x))_{i=p_{r_n}}^{p_{r_{n+1}}-1} \;
   \| \; (A_{k+i})_{i=p_{r_n}}^{p_{r_{n+1}}-1}>_v.$ Therefore, $\vt\leq_k \vs$.

\begin{proof}[Proof of Claim 2]
Fix $n\in\nn$ and let $j\in\nn$ be arbitrary.  By (W2) of Lemma
\ref{unccl18}, we have that
 $$w_{r_n+j}(x) \in
<(s_i(x))_{i=p_{r_n+j}}^{p_{{r_n+j}+1}-1} \; \| \;
(A_{k+i})_{i=p_{r_n+j}}^{p_{{r_n+j}+1}-1}>_v.$$ Moreover, by (W1)
of Lemma \ref{unccl18} and the monotonicity of the sequence
$(k_i)_{i=0}^\infty$, we have
$$k+p_{r_n+j} \geq k_{r_n+j}\geq k_{r_n}+j$$ and therefore, since the sequence of alphabets $(A_i)_{i=0}^\infty$
is increasing, we get that  $$A_{k_{r_n}+j}\subseteq
A_{k+p_{r_n+j}}.$$ Hence,
$$\{w_{r_n+j}(a): a\in A_{k_{r_n}+j}\} \subseteq
<(s_i(x))_{i=p_{r_n+j}}^{p_{{r_n+j}+1}-1} \; \| \;
(A_{k+i})_{i=p_{r_n+j}}^{p_{{r_n+j}+1}-1}>_c,$$ for every
$j\in\nn$. Therefore, for every $d\in\nn$, we conclude that
$$<(w_{r_n+j}(x))_{j=0}^{d}\; \|\;
(A_{k_{r_n}+j})_{j=0}^{d}>_v
 \subseteq <(s_i(x))_{i=p_{r_n}}^{p_{{r_n+d}+1}-1} \; \| \;
(A_{k+i})_{i=p_{r_n}}^{p_{{r_n+d}+1}-1}>_v.$$
 Setting
$d=r_{n+1}-r_n-1$ and using (T1) of Claim 1, the result follows.
\end{proof}
We are now ready to complete the proof. By Claim 2 we have that
$\vt\leq_k\vs$. Moreover, since $(k_n)_{n=0}^\infty$ and
$(p_n)_{n=0}^\infty$ are strictly increasing, we get that $A_{k+i}
\subseteq A_{k_{r_i}}$, for all $i \in \nn$. Therefore, for every
$n\in\nn$, we have
$$ <(t_i(x))_{i=0}^n \; \| \; (A_{k+i})_{i=0}^n>_c \subseteq
  <(t_i(x))_{i=0}^n \; \| \; (A_{k_{r_{i}}})_{i=0}^n>_c,$$
 and so, by  (T2) of Claim 1, we get that
$<(t_i(x))_{i=0}^n \; \| \;
  (A_{k+i})_{i=0}^n>_c \subseteq E$, for every $n\in\nn$. Hence,
 $<\vec{t} \;\|\; ( A_{k+n})_{n=0}^ \infty>_c
\subseteq E$, as desired.
\end{proof}

\begin{proof}[\textbf{Proof of Theorem \ref{thm3}}] Let $r \geq 2 $ and
let $W(A) = \cup _{i=1}^r E_i$. Let $\vec{v}=(x,x,...)$. Then we have that
 $<\vec{v} \; \| \; (A_n)_{n=0}^ \infty >_c \subseteq  W(A)=\cup_{i=1}^r E_i$.
Trivially,  $\cup_{i=1}^r E_i$ is
  $0$-large in
  $\vec{v}$.
Hence, by Fact \ref{Cl30}, there
  exist $1 \leq i \leq r$ and $\vs \leq _0 \vec{v} $
  such that $E_i$ is $0$-large in $\vs$. Applying Corollary
  \ref{unccl19} for $k=0$, we have that
  there exists
   $\vec{t} \leq_0 \vs$
   such that
    $<\vec{t} \; \| \; (A_n)_{n=0}^\infty >_c \subseteq E_i$ and the proof is complete.
\end{proof}

\begin{rem}
In \cite[Theorem 2.3]{MC} the following  version  of Theorem
\ref{thm3f} was shown.

\begin{thm}\label{unconditional left}
  Let  $A$ be a finite alphabet.
 Then for every finite coloring of $W(A)$ there exists a sequence
  $(t_n(x))_{n=0}^ \infty$ of variable words over $A$
  such that  for every $n \geq 1$, $t_n(x)$ is a left variable word and
    for every $n\in\nn$ and every  $0=m_0<m_1<...<m_n$,
  the  words of the form $t_{m_0}(a_0) t_{m_1}(a_1) ... t_{m_n}(a_n)$ with
  $a_{i} \in A$ for all $0 \leq i \leq n$ are of the same color.
\end{thm}
 The above theorem is  a stronger version
of a well-known result of T. Carlson and S. Simpson \cite[Theorem
6.3]{C-S}. We mention also that a left variable version of Theorem
\ref{thm3} does not hold true (see \cite[\S 3]{HMC}). Although our
approach can be applied for \cite[Theorem 6.3]{C-S} (see
\cite{Kar}), it is open for us whether it can also provide an
alternative proof of Theorem \ref{unconditional left}.
\end{rem}

\end{document}